\documentclass[11pt,reqno]{amsart}
\usepackage{amsmath,amsfonts,amssymb,amsthm,fancyhdr,bm}
\usepackage[hidelinks]{hyperref}
\usepackage{fullpage}
\usepackage[dvipsnames]{xcolor}
\usepackage[color=Purple!50!white]{todonotes}

\title[Canonical polynomial van der Waerden theorem]{A short proof of \\ the canonical polynomial van der Waerden theorem}

\author{Jacob Fox}
\author{Yuval Wigderson}
\address{Fox, Wigderson: Department of Mathematics, Stanford University, Stanford, CA, USA}
\email{\{jacobfox,yuvalwig\}@stanford.edu}

\author{Yufei Zhao}
\address{Zhao: Department of Mathematics, Massachusetts Institute of Technology, Cambridge, MA, USA} 
\email{yufeiz@mit.edu}

\thanks{Fox is supported by a Packard Fellowship and by NSF award DMS-1855635.
Wigderson is supported by NSF GRFP grant DGE-1656518. 
Zhao is supported by NSF award DMS-1764176, the MIT Solomon Buchsbaum Fund, and a Sloan Research
Fellowship.}

\newtheorem{thm}{Theorem}
\newtheorem{lem}[thm]{Lemma}

\theoremstyle{definition}

\theoremstyle{remark}

\newcommand\N{\mathbb{N}}
\newcommand\Z{\mathbb{Z}}

\newcommand{\sabs}[1]{\lvert#1\rvert}
\newcommand{\abs}[1]{\left\lvert#1\right\rvert}
\newcommand{\wh}[1]{\widehat{#1}}

\newcommand\dd{\,\mathrm{d}}

\begin{document}

\begin{abstract}
	We present a short new proof of the canonical polynomial van der Waerden theorem, recently established by Gir\~ao.
\end{abstract}

\maketitle

Gir\~ao \cite{Girao} recently proved the following canonical version of the polynomial van der Waerden theorem. Here a set is \emph{rainbow} if all elements have distinct colors. We write $[N]:=\{1, \dots, N\}$.

\begin{thm}[\cite{Girao}]\label{thm:girao}
	Let $p_1,\ldots,p_k$ be distinct polynomials with integer coefficients and $p_i(0)=0$ for each $i$. For all sufficiently large $N$, every coloring of $[N]$ contains a sequence $x+p_1(y),\ldots,x+p_k(y)$  (for some $x,y \in \N$) that is monochromatic or rainbow. 
\end{thm}

Gir\~ao's proof uses a color-focusing argument.
Here we give a new short proof of Theorem~\ref{thm:girao}, deducing it from the polynomial Szemer\'edi's theorem of Bergelson and Leibman \cite{BeLe}.

\begin{thm}[\cite{BeLe}]\label{thm:bele}
	Let $p_1,\ldots,p_k$ be distinct polynomials with integer coefficients and $p_i(0)=0$ for each $i$. Let $\varepsilon>0$. For all $N$ sufficiently large, every $A \subset [N]$ with $|A| \geq \varepsilon N$ contains $x+p_1(y),\ldots,x+p_k(y)$ for some $x,y \in \N$.
\end{thm}

Our proof of Theorem~\ref{thm:girao} follows the strategy of Erd\H os and Graham \cite{ErGr}, who deduced a canonical van der Waerden theorem (i.e., for arithmetic progressions) using Szemer\'edi's theorem~\cite{Szemeredi75}.

We quote the following result, proved by Linnik \cite{Linnik} in his elementary solution of Waring's problem (see \cite[Theorem~19.7.2]{Hua}).
Note the left-hand side below counts the number of solutions $f(x_1) + \cdots + f(x_{s/2}) = f(x_{s/2+1}) + \cdots + f(x_s)$ with $x_1, \dots, x_s \in [n]$.

\begin{thm}[\cite{Linnik}]  \label{thm:moment}
Fix a polynomial $f$ of degree $d \geq 2$ with integer coefficients.
Let $s = 8^{d-1}$. Then
	\[
	\int_0^1 \abs{\sum_{x=1}^n e^{2\pi i \theta f(x)}}^{s} \dd\theta = O(n^{s-d}).
	\]
\end{thm}

\begin{lem}\label{lem:few-edges-local}
    Fix a polynomial $f$ of degree $d \ge 2$ with integer coefficients.
    For every $A \subset \Z$ and $n \in \N$, the number of pairs $(x,y) \in A \times [n]$ with $x+f(y) \in A$ is $O( \abs{A}^{1 + \frac{1}{s}} n^{1-\frac{d}{s}})$, where $s = 8^{d-1}$.
\end{lem}

\begin{proof}
    We write
    \[
	\wh 1_A(\theta) = \sum_{x \in A} e^{2\pi i \theta x} \qquad \text{ and } \qquad F(\theta) = \sum_{y=1}^n e^{2\pi i \theta f(y)}.
	\]
	Then the number of solutions to $z = x+f(y)$ with $x,z \in A$ and $y \in [n]$ is
	\begin{align*}
	\int_0^1 \sabs{\wh 1_A(\theta)}^2 F(\theta) \dd\theta
	&\le 
	\left(\int_0^1 \sabs{\wh 1_A(\theta)}^{\frac{2s}{s-1}} \dd\theta\right)^{1-\frac{1}{s}} 
	\left(\int_0^1 \sabs{F(\theta)}^{s} \dd\theta\right)^{\frac{1}{s}}
	&&\text{\small [H\"older]}
	\\
	&\le  \left( \abs A^{\frac{2}{s-1}} \int_0^1 \sabs{\wh 1_A(\theta)}^2\dd \theta \right)^{1-\frac{1}{s}} \cdot O(n^{1-\frac{d}{s}})
	&&\text{\small [$\sabs{\wh 1_A(\theta)} \leq \abs A$ and Theorem \ref{thm:moment}]}
	\\
	&=  
	\left(\abs{A}^{\frac{2}{s-1}} \abs{A} \right)^{1-\frac{1}{s}}
	\cdot O(n^{1-\frac{d}{s}}) 
	&& \text{\small [Parseval]}
	\\
	&= O(\abs{A}^{1 + \frac{1}{s}} n^{1-\frac{d}{s}}).  &&\qedhere
	\end{align*}
\end{proof}

\begin{lem}\label{lem:few-edges-global}
    Fix a polynomial $f$ of degree $d\ge 1$ with integer coefficients.
    Let $A \subset \Z$.
    Suppose that $|A \cap [x, x +L)| \le \varepsilon L$ for every $L \ge n^d$ and $x$.
    Then the number of pairs $(x,y) \in A \times [n]$ with $x+f(y) \in A$ is $O( \varepsilon^{1/s}  \abs{A} n)$, where $s = 8^{d-1}$.
\end{lem}

\begin{proof}
    If $d = 1$, then for every $x \in A$, the number of $y \in [n]$ so that $x + f(y) \in A$ is $O(\varepsilon n)$ by the local density condition on $A$. Summing over all $x \in A$ yields the desired bound $O(\varepsilon |A| n)$ on the number of pairs. From now on assume $d \ge 2$.
    
    Let $m = O(n^d)$ so that $|f(y)| \le m$ for all $y \in [n]$. 
    Let $A_i = A \cap [im, (i+2) m)$. 
    Then $|A_i| = O(\varepsilon m)$.
    Every pair $x,x+f(y) \in A$ with $y \in [n]$ is contained in some $A_i$, 
    and, by Lemma~\ref{lem:few-edges-local}, the number of pairs contained in each $A_i$ is $O(|A_i|^{1+\frac{1}{s}} n^{1-\frac{d}{s}}) = O((\varepsilon m)^{\frac{1}{s}} |A_i| n^{1-\frac{d}{s}}) = O(\varepsilon^{1/s} |A_i| n)$. Summing over all integers $i$ yields the lemma (each element of $A$ lies in precisely two different $A_i$'s).    
\end{proof}

\begin{proof}[Proof of Theorem \ref{thm:girao}]
    Choose a sufficiently small $\varepsilon > 0$ (depending on $p_1, \dots, p_k$). Consider a coloring of $[N]$ without monochromatic progressions $x+p_1(y), \dots, x+p_k(y)$. By Theorem~\ref{thm:bele}, every color class has density at most $\varepsilon$ on every sufficiently long interval. 
    
    Let $D = \max_{i\ne j} \deg (p_i - p_j)$. 
    Let $n$ be an integer on the order of $N^{1/D}$ so that $x+p_1(y), \dots, x+p_k(y) \in [N]$ only if $y \in [n]$. 
    For each color class $A$, applying Lemma \ref{lem:few-edges-global} to $f = p_i - p_j$ and summing over all $i\ne j$,
    we see that the number of pairs $(x,y) \in \Z \times [n]$ where at least two of $x + p_1(y), \dots, x + p_k(y)$ lie in $A$ is  $O(\varepsilon^{1/8^{D-1}} |A| n)$. Summing over all color classes $A$, we see that the number of non-rainbow progressions $x + p_1(y), \dots, x + p_k(y) \in [N]$ is $O(\varepsilon^{1/8^{D-1}} Nn)$. Since the total number of sequences $x + p_1(y), \dots, x + p_k(y) \in [N]$ is on the order of $Nn$, some such sequence must be rainbow, as long as $\varepsilon > 0$ is small enough and $N$ is large enough.
\end{proof}

\end{document}